\documentclass[12pt]{amsart}

\usepackage{a4wide}
\usepackage{amsmath, amssymb}
\usepackage{amsthm}
\usepackage[utf8]{inputenc}
\usepackage{subfigure}
\usepackage{enumerate}
\usepackage{cancel}
\usepackage{xspace}
\usepackage{float}

\usepackage{mathrsfs}
\usepackage[colorlinks, citecolor=blue, anchorcolor=blue, linkcolor=blue, urlcolor=blue]{hyperref}
\usepackage{color,xcolor}

\usepackage{wrapfig}

\usepackage{tikz}
\usetikzlibrary{patterns}

\usepackage{algorithm}
\usepackage{algorithmic}

\def\SYoung#1{\vbox{\smallskip\offinterlineskip
    \halign{&\vbox{##}\kern-\SThickness\cr #1}}}

\newdimen\SSquaresize \SSquaresize=4.5pt
\newdimen\SThickness \SThickness=.15pt
\newdimen\SCorrection \SCorrection=7pt

\def\SCarre#1{\hbox{\vrule width \SThickness
   \vbox to \SSquaresize{\hrule height \SThickness\vss
      \hbox to \SSquaresize{\hss$\scriptstyle#1$\hss}
   \vss\hrule height\SThickness}
   \unskip\vrule width \SThickness}
   \kern-\SThickness}

\allowdisplaybreaks

\makeatletter \@addtoreset{equation}{section}

\newtheorem{theorem}{Theorem}[section]

\newtheorem{remark}[theorem]{Remark}

\title[The Briggs inequality of Boros-Moll sequences]{The Briggs inequality of Boros-Moll sequences}

\author[Z.-X. Zhang]{Zhong-Xue Zhang}
       \address{Center for Combinatorics, LPMC, Nankai University, Tianjin 300071, P. R. China.}
       \email{zhzhx@mail.nankai.edu.cn}

\author[James J. Y. Zhao]{James Jing Yu Zhao}
       \address{School of Accounting, Guangzhou College of Technology and Business,
       Foshan 528138, P. R. China.}
       \email{zhao@gzgs.edu.cn}



\begin{document}

\begin{abstract}
Briggs conjectured that if a polynomial $a_0+a_1x+\cdots+a_nx^n$ with real coefficients has only negative zeros, then
$$a^2_k(a^2_k - a_{k-1}a_{k+1}) > a^2_{k-1}(a^2_{k+1} - a_ka_{k+2})
$$
for any $1\leq k\leq n-1$. The Boros-Moll sequence $\{d_i(m)\}_{i=0}^m$ arises in the study of evaluation of certain quartic integral, and a lot of interesting inequalities for this sequence have been obtained.
In this paper we show that the Boros-Moll sequence $\{d_i(m)\}_{i=0}^m$, its normalization $\{d_i(m)/i!\}_{i=0}^m$, and its transpose $\{d_i(m)\}_{m\ge i}$ satisfy the Briggs inequality.
For the first two sequences, we prove the Briggs inequality
by using a lower bound for $(d_{i-1}(m)d_{i+1}(m))/d_i^2(m)$ due to Chen and Gu and an upper bound due to Zhao.
For the transposed sequence, we derive the Briggs inequality by establishing its strict ratio-log-convexity. As a consequence, we also obtain the strict log-convexity of the sequence $\{\sqrt[n]{d_i(i+n)}\}_{n\ge 1}$ for $i\ge 1$.
\end{abstract}

\maketitle

\section{Introduction}

In the study of positive irreducibility of binding polynomials, Briggs \cite{Briggs1985-2} proposed two inequality conjectures concerning elementary symmetric functions over a set of positive numbers. Suppose that $X=\{x_1,x_2,\ldots,x_n\}$ is a set of $n$ variables. Recall that the $k$-th elementary symmetric function is defined by
\begin{align*}
	e_k(X)=\sum_{1\leq j_1<j_2<\cdots<j_k\leq n}x_{j_1}x_{j_2}\cdots x_{j_k}, \qquad 1\leq k\leq n.
\end{align*}
By convention, we set $e_0=1$ and $e_k=0$ if $k<0$ or $k>n$.
Briggs \cite{Briggs1985-2} conjectured that if $X$ is a set of $n$ positive numbers, then for $1\leq k\leq n-1$ the following two inequalities hold:
\begin{align}
	e_{k-1}e^2_{k+2} + e^2_{k}e_{k+3} + e^3_{k+1} &> e_{k+1}\left( e_{k-1}e_{k+3} + 2e_{k}e_{k+2}\right),\label{equ-elementary-1}\\
	e^2_k(e^2_k - e_{k-1}e_{k+1}) &> e^2_{k-1}(e^2_{k+1} - e_ke_{k+2}).\label{equ-elementary-2}
\end{align}
It is worth noting that the expression
\begin{align*}
	(x+x_1)(x+x_2)\cdots (x+x_n) =\sum_{k=0}^n e_{n-k}x^k
\end{align*}
allows us to formulate Briggs' conjecture in the following manner: If a polynomial $f(x)=\sum_{k=0}^n a_kx^k$ with positive coefficients has only real zeros, then for $1\leq k\leq n-1$ we have
\begin{align}
	a_{k-1}a^2_{k+2} + a^2_{k}a_{k+3} + a^3_{k+1} &> a_{k+1}\left( a_{k-1}a_{k+3} + 2a_{k}a_{k+2}\right),\label{equ-Briggs-1}\\
	a^2_k(a^2_k - a_{k-1}a_{k+1}) &> a^2_{k-1}(a^2_{k+1} - a_ka_{k+2}).\label{equ-Briggs-2}
\end{align}

We would like to point out that the inequality \eqref{equ-Briggs-1} is closely related to the notion of $2$-log-concavity. Recall that a sequence $\{a_k\}_{k\ge 0}$ with real numbers is said to be \emph{log-concave} if
\begin{align}\label{eq:df-log-concave}
a_k^2-a_{k-1}a_{k+1}\ge 0
\end{align}
for any $k\ge 1$.
The sequence $\{a_k\}_{k\ge 0}$ is called log-convex if $a_k^2-a_{k-1}a_{k+1}\le 0$ for $k\ge 1$.
A polynomial is said to be log-concave (resp. log-convex) if its coefficient sequence is log-concave (resp. log-convex). For more information on log-concave and log-convex sequences, see Br\"{a}nd\'{e}n \cite{Branden2015}, Brenti \cite{Brenti} and Stanley \cite{Stanley}.
For a sequence $\{a_n\}_{n\geq 0}$, define an operator $\mathcal{L}$ by $\mathcal{L}(\{a_n\}_{n\geq 0})=\{b_n\}_{n\geq 0}$, where $b_n=a_n^2-a_{n-1}a_{n+1}$ for $n\geq 0$, with the convention that $a_{-1}=0$.
A sequence $\{a_n\}_{n\geq 0}$ is said to be $k$-log-concave (resp. strictly $k$-log-concave) if the sequence $\mathcal{L}^j\left(\{a_n\}_{n\geq 0}\right)$ is nonnegative (resp. positive) for each $1\le j\le k$, and $\{a_n\}_{n\geq 0}$ is said to be $\infty$-log-concave if $\mathcal{L}^k\left(\{a_n\}_{n\geq 0}\right)$ is nonnegative for any $k\ge 1$.
The (strict) $k$-log-convexity is defined in a similar manner.
The notion of infinite log-concavity was introduced by Boros and Moll \cite{Boros-Moll-2004}.
We see that \eqref{equ-Briggs-1} can be reformulated as
\begin{align}\label{equ-Griggs-1-reformulation}
	\left|\begin{matrix}
		a_{k+1} & a_{k+2} & a_{k+3}\\
		a_{k} & a_{k+1} & a_{k+2}\\
		a_{k-1} & a_{k} & a_{k+1}
	\end{matrix}\right|>0.
\end{align}
By Newton's inequality, if $f(x)=\sum_{k=0}^n a_kx^k$ has only real zeros, then the sequence $\{a_k\}_{k\geq 0}$ is strictly
log-concave. Due to the positivity and log-concavity of the coefficients $a_k$, the inequality \eqref{equ-Griggs-1-reformulation}
amounts to say that the sequence $\{a_k\}_{k\geq 0}$ is 2-log-concave.
It is also natural to study \eqref{equ-Briggs-1} and \eqref{equ-Briggs-2} for infinite sequences $\{a_k\}_{k\geq 0}$.
Notably, the 2-log-concavity property was verified for various sequences. For instance, Hou and Zhang \cite{HZ-19} and Jia and Wang \cite{JW-20} independently proved it for the partition function $\{p(i)\}_{i\geq 222}$.
Mukherjee \cite{Mukherjee-22}  explored  the sequence of overpartition functions  $\{\overline{p}(i)\}_{i\geq 42}$.
Yang \cite{Yang-23-12} investigated the sequences of differences $\{p(i)-p(i-1)\}_{i\geq 347}$ and $\{\overline{p}(i)-\overline{p}(i-1)\}_{i\geq 64}$. Yang \cite{Yang-23-5} also examined the sequence of broken $k$-diamond partition function $\{\Delta_k(i)\}_{i\geq 20}$ for $k=1$ or $2$.

It is known that a real-rooted polynomial with nonnegative coefficients, more generally a multiplier sequence, is 2-log-concave due to Aissen, Edrei,  Schoenberg and Whitney \cite{AESW-51}. See also Craven and Csordas \cite[Theorem 2.13]{Craven-Csordas1989} for the case of multiplier sequence.
Hence, the conjectured inequality \eqref{equ-Briggs-1} for a real-rooted polynomial with positive coefficients is already established by allowing equality.
In fact, the strict inequality  can also be proved by using Newton's inequality and Br\"and\'en's theorem \cite{Branden} (see also \cite[Theorem 1.2]{Chen-Dou-Yang}).

However, the conjectured inequality \eqref{equ-Briggs-2} for a real-rooted polynomial with positive coefficients is still open. From now on we will call  \eqref{equ-Briggs-2} the \emph{Briggs inequality}.
Recently, Liu and Zhang \cite{Liu-Zhang} proved the Briggs inequality for both the partition function sequence $\{p(i)\}_{i\geq 114}$ and the overpartition function sequence $\{\overline{p}(i)\}_{i\ge 18}$. The objective of this paper is to prove the Briggs inequality for the Boros-Moll
sequence, the normalized Boros-Moll sequence and the transposed Boros-Moll sequence, which we shall recall below.

The Boros-Moll
sequences arise in the study of the following  quartic integral
$$
\int_0^\infty\frac{1}{(t^4+2xt^2+1)^{m+1}} dt
$$
for $x>-1$ and $m\in\mathbb{N}$.
Boros and Moll \cite{Boros-Moll-1999a, Boros-Moll-2001}
proved that this integral is equal to
$\frac{\pi}{2^{m+3/2}(x+1)^{m+1/2}} P_m(x)$, where
\begin{align}\label{eq:B-M-poly2}
P_m(x)
=2^{-2m}\sum_{k=0}^m 2^k \binom{2m-2k}{m-k}\binom{m+k}{k} (x+1)^k.
\end{align}

The coefficients of $x^i$ in $P_m(x)$, denoted by $d_i(m)$, are called the Boros-Moll numbers. The polynomials $P_m(x)$ are called the Boros-Moll polynomials, and the sequences $\{d_i(m)\}_{i=0}^m$ are called the Boros-Moll sequences. Clearly, one sees from \eqref{eq:B-M-poly2} that
\begin{align}\label{eq:B-M-seq}
d_i(m)=2^{-2m}\sum_{k=i}^m 2^k \binom{2m-2k}{m-k}\binom{m+k}{k}\binom{k}{i}
\end{align}
for $0\le i \le m$.
See \cite{Amdeberhan-Moll, Boros-Moll-1999b, Boros-Moll-1999c, Boros-Moll-2001, Boros-Moll-2004, Chen-Gu2009, Chen-Wang-Xia, Guo2022, HMY, Moll2002} for more information on these sequences.

Boros and Moll \cite{Boros-Moll-1999b} showed that the sequence $\{d_i(m)\}_{i=0}^m$ is unimodal with the maximum term located in the middle for $m\ge 2$,
see also \cite{AMABKMR, Boros-Moll-1999c}.
Moll \cite{Moll2002} conjectured a stronger property that the sequences $\{d_i(m)\}_{i=0}^m$ are log-concave, which was proved by Kauers and Paule \cite{Kauers-Paule} with a computer algebra method. Chen, Pang and Qu \cite{Chen-Pang-Qu-2012}  also gave a combinatorial proof for this conjecture by building a structure of partially $2$-colored permutations.
Boros and Moll \cite{Boros-Moll-2004} also conjectured that the sequence $\{d_i(m)\}_{i=0}^m$ is $\infty$-log-concave, and this conjecture is still open.

Motivated by Newton's inequality and the infinite log-concavity conjecture on binomial numbers proposed by Boros and Moll \cite{Boros-Moll-2004}, Fisk \cite{Fisk}, McNamara and Sagan \cite{McNamara-Sagan} and Stanley (see \cite{Branden}) independently gave a general conjecture, which states that if a polynomial $\sum_{k=0}^n a_k x^k$ has only real and negative zeros, then so does the polynomial $\sum_{k=0}^n (a_k^2-a_{k-1}a_{k+1}) z^k$ where $a_{-1}=a_{n+1}=0$. This conjecture was confirmed by Br\"{a}nd\'{e}n \cite{Branden}.
Br\"{a}nd\'{e}n's Theorem provides an approach to $\infty$-log-concavity of a sequence by relating real-rooted polynomials to higher-order log-concavity.

Although, as shown by Boros and Moll \cite{Boros-Moll-1999b}, the polynomials $P_m(x)$ are not real-rooted in general, Br\"{a}nd\'{e}n introduced two polynomials
\begin{align*}
Q_m(x)=\sum_{i=0}^m \frac{d_i(m)}{i !}x^i, \quad
R_m(x)=\sum_{i=0}^m \frac{d_i(m)}{(i+2)!}x^i,\quad m\ge 1,
\end{align*}
derived from $P_m(x)$ and conjectured their real-rootedness \cite[Conjectures 8.5 \& 8.6]{Branden}.
Br\"{a}nd\'{e}n's conjectures were proved by Chen, Dou and Yang \cite{Chen-Dou-Yang}.
As noted by Br\"{a}nd\'{e}n \cite{Branden}, based on two theorems of Craven and Csordas \cite{Craven-Csordas2002} on iterated Tur\'{a}n inequalities, the real-rootedness of $Q_m(x)$ and $R_m(x)$ imply, respectively, the $2$-log-concavity and the $3$-log-concavity of $P_m(x)$.
In another direction, Chen and Xia \cite{Chen-Xia2} showed an analytic proof of the $2$-log-concavity of the Boros-Moll sequences by founding an intermediate function
so that the quartic inequalities for the $2$-log-concavity is reduced to quadratic inequalities.

Since $d_i(m)$ has two parameters, it is natural to consider properties for the sequences $\{d_i(m)\}_{m\ge i}$, which were called \emph{transposed Boros-Moll sequences} in \cite{Zhao2023-2} where the extended ultra log-concavity and the asymptotic ratio-log-convexity for $\{d_i(m)\}_{m\ge i}$ were proved, and the asymptotic log-convexity for $\{\sqrt[n]{d_i(i+n)}\}_{n\ge i^2}$ was also obtained. Note that recently, Jiang and Wang \cite{Jiang-Wang} proved the $2$-log-concavity, the higher order Tur\'{a}n inequality and the Laguerre inequality of lower order for the sequences $\{d_i(m)\}_{m\ge i}$.

As mentioned above, both
the Boros-Moll sequence $\{d_i(m)\}_{i=0}^m$ and its transpose $\{d_i(m)\}_{m\ge i}$ are 2-log-concave, then both of them satisfy \eqref{equ-Briggs-1}. This motivated us to study whether they also satisfy the Briggs inequality \eqref{equ-Briggs-2}. On the other hand, since
Briggs' conjecture on \eqref{equ-Briggs-2} is still open for real polynomials with only negative zeros, it is also desirable to know whether $Q_m(x)$ and $R_m(x)$ satisfy
the Briggs inequality. For this purpose, we define the \emph{normalized Boros-Moll sequence} $\{\tilde{d}_{i,k}(m)\}_{i=1}^m$ as $\{{d}_i(m)/(i+k)!\}_{i=1}^m$ for any $k\geq 0$.
Our main results are stated as follows.

\begin{theorem}\label{Thm:Briggs-ieq}
For each $m\ge 2$, the Boros-Moll sequence $\{d_i(m)\}_{i=1}^m$ satisfies the Briggs inequality. That is, for any $m\ge 2$ and $1\le i \le m-1$,
\begin{align}\label{ieq:Briggs-BM}
  d_{i}^2(m)(d_{i}^2(m)-d_{i-1}(m)d_{i+1}(m))
>d_{i-1}^2(m)(d_{i+1}^2(m)-d_{i}(m)d_{i+2}(m)).
\end{align}
\end{theorem}

\begin{theorem}\label{Thm:Briggs-ieq-k}
For any $k\geq 0$ and $m\ge 2$, the normalized Boros-Moll sequence $\{\tilde{d}_{i,k}(m)\}_{i=1}^m$ satisfies the Briggs inequality. That is,
for any $k\geq 0$, $m\ge 2$ and $1\le i \le m-1$,
\begin{align}\label{ieq:Briggs-BM-k}
  \tilde{d}_{i,k}^2(m)\big(\tilde{d}_{i,k}^2(m)
   -\tilde{d}_{i-1,k}(m)\tilde{d}_{i+1,k}(m)\big)
 >\tilde{d}_{i-1,k}^2(m)\big(\tilde{d}_{i+1,k}^2(m)
   -\tilde{d}_{i,k}(m)\tilde{d}_{i+2,k}(m)\big).
\end{align}
\end{theorem}

\begin{theorem}\label{Thm:Briggs-ieq-tr}
For each $i\ge 1$, the transposed Boros-Moll sequence $\{d_i(m)\}_{m\ge i}$ satisfies the Briggs inequality. That is, for any $i\ge 1$ and $m\ge i+1$,
\begin{align}\label{ieq:Briggs-trBM}
  d_{i}^2(m)(d_{i}^2(m)-d_{i}(m-1)d_{i}(m+1))
>d_{i}^2(m-1)(d_{i}^2(m+1)-d_{i}(m)d_{i}(m+2)).
\end{align}
However, for $i=0$, the inverse relation of \eqref{ieq:Briggs-trBM} holds.
\end{theorem}

In the remainder of this paper, we first complete the proofs of Theorems \ref{Thm:Briggs-ieq} and \ref{Thm:Briggs-ieq-k}, in Section \ref{Sec:2}, by employing a lower bound for $(d_{i-1}(m)d_{i+1}(m))/d_i^2(m)$ given by Chen and Gu \cite{Chen-Gu2009} and an upper bound given by Zhao \cite{Zhao2023}.
In Section \ref{Sec:3} we establish the strict ratio-log-convexity of the transposed Boros-Moll sequence $\{d_i(m)\}_{m\ge i}$, which is the key ingredient of the proof of Theorem \ref{Thm:Briggs-ieq-tr}.
The complete proof of Theorem \ref{Thm:Briggs-ieq-tr} will be presented in Section \ref{Sec:4}.
At last, in Section \ref{Sec:5}, we derive the strict log-convexity of the sequence $\{\sqrt[n]{d_i(i+n)}\}_{n\ge 1}$ for each $i\ge 1$ from the strict ratio-log-convexity of $\{d_i(m)\}_{m\ge i}$.

\section{Proofs of Theorems \ref{Thm:Briggs-ieq} and \ref{Thm:Briggs-ieq-k}}\label{Sec:2}

This section is devoted to  proving the Briggs inequality for the Boros-Moll sequence $\{{d}_i(m)\}_{i=1}^m$ and the normalized Boros-Moll sequence $\{\tilde{d}_{i,k}(m)\}_{i=1}^m$.

Let us first prove Theorem \ref{Thm:Briggs-ieq}.
In order to do so, we need proper bounds for $(d_{i-1}(m)d_{i+1}(m))/d_i^2(m)$. Chen and Gu \cite{Chen-Gu2009} established the following inequality in studying the reverse ultra log-concavity of the Boros-Moll polynomials.
\begin{theorem}$($\cite[Theorem 1.1]{Chen-Gu2009}$)$\label{thm:Chen-Gu-1}
For $m\ge 2$ and $1\le i \le m-1$, there holds
\begin{align}\label{eq:ub-mell2}
\frac{d_i(m)^2}{d_{i-1}(m)d_{i+1}(m)}
<\frac{(m-i+1)(i+1)}{(m-i)i}.
\end{align}
\end{theorem}

The following result was obtained by Zhao \cite{Zhao2023}, which was used to give a new proof of the higher order Tur\'{a}n inequality for the Boros-Moll sequence.

\begin{theorem}$($\cite[Theorem 3.1]{Zhao2023}$)$\label{thm:sharper-bd}
For each $m\ge 2$ and $1\le  i \le m-1$, we have
\begin{align}\label{eq:lb-mell1}
\frac{d_i^2(m)}{d_{i-1}(m)d_{i+1}(m)}
>\frac{(m-i+1)(i+1)(m+i^2)}{(m-i)i(m+i^2+1)}.
\end{align}
\end{theorem}

We are now in a stage to show a proof of Theorem \ref{Thm:Briggs-ieq}.

\begin{proof}[Proof of Theorem \ref{Thm:Briggs-ieq}]
By \eqref{eq:B-M-seq}, it is clear that $d_i(m)>0$ for $m\ge 1$ and $1\le i \le m$. So, for any $m\ge 2$ and $1\le i \le m-1$, the inequality \eqref{ieq:Briggs-BM} can be rewritten as
\begin{align}\label{ieq:Briggs-BM-re}
 \big(u_{i+1}(m)-1\big)u_i(m)-1+\frac{1}{u_i(m)}>0,
\end{align}
where
\begin{align}\label{defi:uim}
 u_i(m)=\frac{d_{i-1}(m)d_{i+1}(m)}{d_i^2(m)}.
\end{align}

We aim to prove \eqref{ieq:Briggs-BM-re} for $m\ge 2$ and $1\le i \le m-1$. As will be seen, the bounds given by Theorems \ref{thm:Chen-Gu-1} and \ref{thm:sharper-bd} are crucial to the proof.

For any $1\le i\le m-1$ and $m\ge 2$, it is routine to verify that
$$(m-i+1)(i+1)(m+i^2)>(m-i)i(m+i^2+1)$$
and hence by \eqref{eq:lb-mell1} we have
\begin{align}\label{u_i<1}
 0< u_i(m) <1.
\end{align}
It should be noted that \eqref{u_i<1} is also implied by
a result due to Chen and Xia \cite[Theorem 1.1]{Chen-Xia}.

By Theorems \ref{thm:Chen-Gu-1} and \ref{thm:sharper-bd}, we have for $m\ge 2$ and $1\le i \le m-1$,
\begin{align}\label{bds-uim}
 f_i(m)<u_i(m)<g_i(m),
\end{align}
where
\begin{align}\label{defi:figim}
f_i(m)=\frac{(m-i)i}{(m-i+1)(i+1)},\quad
g_i(m)=\frac{(m-i)i(m+i^2+1)}{(m-i+1)(i+1)(m+i^2)}.
\end{align}

It follows from \eqref{u_i<1} and \eqref{bds-uim} that for $m\ge 2$ and $1\le i \le m-2$,
\begin{align}\label{ieq:Briggs>}
 (u_{i+1}(m)-1)u_i(m)-1+\frac{1}{u_i(m)}
>(f_{i+1}(m)-1)g_i(m)-1+\frac{1}{g_i(m)}.
\end{align}
Denote by $E_i(m)$ the right-hand side of \eqref{ieq:Briggs>}. Observe that
\begin{align}\label{ieq:B-r}
 E_i(m)
=\frac{A}{i(i+1)(i+2)(m-i)(m-i+1)(m+i^2)(m+i^2+1)},
\end{align}
where
\begin{align*}
 A
=&\ i^5(2m^2-i^2+10m-4i)+5i^3m^2(m-i)+i^3m(13m-3i)+3im^3(m-i)\\
 &\ +4i^5+4i^4+4i^3m+4i^2m^2+2im^3+2m^4+i^3+5i^2m+m^2(4m-i)+2m^2.
\end{align*}
For $m\ge 2$ and $1\le i \le m-2$, it is clear that $A>0$, and hence $E_i(m)>0$. Then we obtain \eqref{ieq:Briggs-BM-re} by \eqref{ieq:Briggs>} for $m\ge 2$ and $1\le i \le m-2$.

It remains to prove \eqref{ieq:Briggs-BM-re} for $m\ge 2$ and $i=m-1$. In this case, note that
\begin{align*}
 u_{i+1}(m)=u_m(m)
=\frac{d_{m-1}(m)d_{m+1}(m)}{d_m^2(m)}
=f_{i+1}(m)=f_m(m)=0.
\end{align*}
Clearly, the inequality in \eqref{ieq:Briggs>} still holds for $m\ge 2$ and $i=m-1$. It is easily checked that
$E_i(m)>0$ for $m\ge 2$ and $i=m-1$. Again, we arrive at \eqref{ieq:Briggs-BM-re} by \eqref{ieq:Briggs>}.
This completes the proof.
\end{proof}

We proceed to prove Theorem \ref{Thm:Briggs-ieq-k}.

\begin{proof}[Proof of Theorem \ref{Thm:Briggs-ieq-k}]
Fix $k\ge 0$ and $m\ge 2$ throughout this proof. Similar to the proof of Theorem \ref{Thm:Briggs-ieq}, for any $1\le i \le m-1$, we reformulate the inequality \eqref{ieq:Briggs-BM-k} as
\begin{align}\label{ieq:Briggs-BM-re-k}
 \big(u_{i+1,k}(m)-1\big)u_{i,k}(m)-1+\frac{1}{u_{i,k}(m)}>0,
\end{align}
where
\begin{align}\label{defi:uim-k}
u_{i,k}(m)=\frac{i+k}{i+k+1}u_i(m).
\end{align}

We aim to prove \eqref{ieq:Briggs-BM-re-k} for $1\le i \le m-1$.
By \eqref{u_i<1} and \eqref{defi:uim-k}, we see that
\begin{align}\label{u_i_k<1}
 0< u_{i,k}(m) <1.
\end{align}
Moreover, from the relations \eqref{bds-uim} and \eqref{defi:uim-k} we have that for $1\le i \le m-1$,
\begin{align}\label{bds-uim-k}
  f_{i,k}(m)<u_{i,k}(m)<g_{i,k}(m),
\end{align}
where
\begin{align}
f_{i,k}(m)=\frac{i+k}{i+k+1}f_{i}(m),\quad
g_{i,k}(m)=\frac{i+k}{i+k+1}g_i(m).
\end{align}
Combining \eqref{u_i_k<1} and \eqref{bds-uim-k}, we deduce that for $1\le i \le m-2$,
\begin{align}\label{ieq:Briggs>-k}
 \big(u_{i+1,k}(m)-1\big)u_{i,k}(m)-1+\frac{1}{u_{i,k}(m)}
>\big(f_{i+1,k}(m)-1\big)g_{i,k}(m)-1+\frac{1}{g_{i,k}(m)}.
\end{align}
Let $E_{i,k}(m)$ denote the right-hand side of \eqref{ieq:Briggs>-k}. Observe that
{\small
\begin{align*}
 E_{i,k}(m)
=\frac{B_0+B_1\cdot k+B_2\cdot k^2+A\cdot k^3}{i(i+1)(i+2)(m-i)(m-i+1)(m+i^2)(m+i^2+1)(i + k) (i + k+1) (i + k+2) },
\end{align*}
}%
where $A$ remains the same as \eqref{ieq:B-r}, and $B_0$, $B_1$, $B_2$ are defined as follows:
{\small
\begin{align*}
 B_0=
 &\ 9 i^8 m^2+19 i^6 m^3-14 i^6 m^2+42 i^5 m^3+40 i^5 m^2+10 i^4 m^4+29 i^4 m^3+69 i^4 m^2+32 i^3 m^4\\
 &\ +30 i^3 m^3+46 i^3 m^2+37 i^2 m^4+42 i^2 m^3+17 i^2 m^2-10 i^9 m-4 i^8 m+27 i^7 m+15 i^6 m\\
 &\ +7 i^5 m+21 i^4 m+28 i^3 m+8 i^2 m+2 i^{10}-5 i^9-22 i^8-12 i^7+12 i^6+17 i^5+4 i^4\\
 &\ +20 i m^4 +32 i m^3+12 i m^2+4 m^4+8 m^3+4 m^2\\
 B_1=&\ i^7 (17 m^2 - 14 im + 24  m - 24 i) + i^4 (51 m^3 - 38 i^3) +  i^5 m^2 (37 m - 11 i) + 59 i^6 m + 20 i^3 m^4 \\
 &\ + 3 i^5 m^2 + 8 i^6 +  8 i^5 m + 110 i^4 m^2 + 16 i^3 m^3 + 34 i^5 + 25 i^4 m +
 87 i^3 m^2 + 42 i^2 m^3 + 37 i m^4 \\
     &\ + 16 i^4 + 50 i^3 m +
 28 i^2 m^2 + 54 i m^3 + 48 i^2 m^4 + 10 m^4 + 24 i^2 m + 17 i m^2 +
 20 m^3 + 10 m^2\\
 B_2=&\ i^6 (10  m^2 - 4 i m - 3 i^2 + 34  m - 21 i) + i^5 (21  m - 8 i) +  i^4 m^2 (23 m - 14 i) + 30 i^4 m^2 \\
 &\ + 10 i^3 m^3  + 13 i^2 m^4 +  20 i^5 + i^4 m + 62 i^3 m^2 + i^2 m^3 + 20 i m^4+ 15 i^4 +  27 i^3 m + 17 i^2 m^2   \\
 &\ + 24 i m^3 + 8 m^4 + i^3 + 21 i^2 m +  4 i m^2 + 16 m^3 + 8 m^2.
\end{align*}
}%

For any $1\le i \le m-2$, it is clear that $A,\ B_1, \ B_2>0$.
To show $B_0>0$, notice that
\begin{align*}
  B_0 = (m-i)C + D,
\end{align*}
where
\begin{align*}
  C  = &\ i^8 (9 m+15-i) +i^7 (19 m+65) + i^6 \left(19 m^2+38 m+116\right)+i^5 \left(52 m^2+101 m+143\right)\\
  &\ +i^4 \left(10 m^3+61 m^2+136 m+129\right) +i^3 \left(32 m^3+67 m^2+108 m+81\right) \\
  &\ +i^2 \left(37 m^3+62 m^2+53 m+28\right) +4 i \left(5 m^3+9 m^2+5 m+1\right)+4 m (m+1)^2,\\
  D  =&\ 4 i^2 + 28 i^3 + 85 i^4 + 146 i^5 + 155 i^6 + 104 i^7 + 43 i^8 +
  10 i^9 + i^{10}.
\end{align*}
Hence, $E_{i,k}(m)>0$. Then \eqref{ieq:Briggs-BM-re-k} follows from \eqref{ieq:Briggs>-k} for $k\geq 0$, $m\ge 2$ and $1\le i \le m-2$.

It remains to verify \eqref{ieq:Briggs-BM-re-k} for $k\ge 0$, $m\ge 2$ and $i=m-1$. In this case, note that
\begin{align*}
 u_{i+1,k}(m)=u_{m,k}(m)
=\frac{m+k}{m+k+1} u_m(m)
=f_{i+1,k}(m)=f_{m,k}(m)
=\frac{m+k}{m+k+1}f_{m}(m)=0.
\end{align*}
Clearly, the inequality \eqref{ieq:Briggs>-k} still holds for $k\geq 0$, $m\ge 2$ and $i=m-1$. It is easy to verify that $E_{m-1,k}(m)>0$ for $k\geq 0$ and $m\ge 2$. Again, we arrive at \eqref{ieq:Briggs-BM-re-k} by \eqref{ieq:Briggs>-k}.
This completes the proof.
\end{proof}

\section{Strict ratio-log-convexity of $\{d_i(m)\}_{m\ge i}$}
\label{Sec:3}

The aim of this section is to prove the strict ratio-log-convexity of transposed Boros-Moll sequences $\{d_i(m)\}_{m\ge i}$ for $i\ge 1$. As will be seen in Section \ref{Sec:4}, this property plays a key role in our proof of Theorem \ref{Thm:Briggs-ieq-tr}.

Chen, Guo and Wang \cite{CGW2014} initiated the study of ratio log-behavior of combinatorial sequences.
Recall that a real sequence $\{a_n\}_{n\ge 0}$ is called ratio log-concave (resp. ratio log-convex) if the sequence $\{a_n/a_{n-1}\}_{n\ge 1}$ is log-concave (resp. log-convex). They also showed that the ratio log-concavity (resp. ratio log-convexity) of a sequence $\{a_n\}_{n\ge N}$ implies the strict log-concavity (resp. strict log-convexity) of the sequence $\{\sqrt[n]{a_n}\}_{n\ge N}$ under certain initial condition  \cite{CGW2014}. By applying these criteria, they confirmed a conjecture of Sun \cite{Sun2013} on the log-concavity of the sequence $\{\sqrt[n]{D_n}\}_{n\ge 1}$ for the Domb numbers $D_n$.

Partial progress has been made on the ratio log-behaviour of the sequence $\{d_i(m)\}_{m\ge i}$. Zhao \cite[Theorem 6.1]{Zhao2023-2} proved that the sequence $\{d_i(m)\}_{m\ge i}$ is strictly ratio log-concave for $i=0$, while it is strictly ratio log-convex for each $1\le i\le 135$. For $i\ge 136$, Zhao \cite[Theorem 6.2]{Zhao2023-2} proved an asymptotic result for this property for $m\ge \lfloor(\sqrt{2}/4)i^{3/2}-15i/32\rfloor+2$.
With a sharper bound for $d_i(m)/d_i(m-1)$, we obtain the following exact result for this property.

\begin{theorem}\label{thm:st-r-lvx}
The transposed Boros-Moll sequences $\{d_i(m)\}_{m\ge i}$ are strictly ratio-log-convex for any $i\ge 1$. That is, for each $i\ge 1$ and $m\ge i+2$,
\begin{align}\label{ieq:r-lg-cvx-BM}
 \left(\frac{d_i(m)}{d_i(m-1)}\right)^2
<\left(\frac{d_i(m-1)}{d_i(m-2)}\right)\left(\frac{d_i(m+1)}{d_i(m)}\right).
\end{align}
\end{theorem}

We shall prove Theorem \ref{thm:st-r-lvx} by applying the following criterion for the ratio log-convexity of a sequence obtained by Sun and Zhao \cite{SunZhao}, which was deduced along with the spirit of Chen, Guo, and Wang \cite[Sectioin 4]{CGW2014}.

\begin{theorem}$($\cite[Theorem 4.2]{SunZhao}$)$\label{thm-crlx}
Let $\{S_n\}_{n\ge 0}$ be a positive sequence satisfying the following recurrence relation,
\begin{align*}
S_n=\mathfrak{a}(n)S_{n-1}+\mathfrak{b}(n)S_{n-2}, \quad n\ge 2,
\end{align*}
with real $\mathfrak{a}(n)$ and $\mathfrak{b}(n)$.
Suppose $\mathfrak{a}(n)>0$ and $\mathfrak{b}(n)<0$ for $n\ge N$ where $N$ is a nonnegative integer. If there exists a function $g(n)$ such that for all $n\ge N+2$,
\begin{itemize}
\item[$(i)$] $\frac{\mathfrak{a}(n)}{2}\le g(n)\le \frac{S_n}{S_{n-1}};$
\item[$(ii)$]  $4g^3(n)-3\mathfrak{a}(n)g^2(n)-\mathfrak{a}(n+1)\mathfrak{b}(n)\ge 0;$
\item[$(iii)$]$g^4(n)-\mathfrak{a}(n)g^3(n)-\mathfrak{a}(n+1)\mathfrak{b}(n)g(n)
  -\mathfrak{b}(n)\mathfrak{b}(n+1)\ge 0,$
\end{itemize}
then $\{S_n\}_{n\ge N}$ is ratio log-convex, that is, for $n\ge N+2$,
\begin{align}\label{ieq:ra-lcv-1}
 (S_n/S_{n-1})^2\le (S_{n-1}/S_{n-2})(S_{n+1}/S_{n}).
\end{align}
\end{theorem}

\begin{remark}\label{rmk}
By the proof of Theorem \ref{thm-crlx}, it is easy to see that the inequality in \eqref{ieq:ra-lcv-1} holds strictly if the inequality in condition (iii) is strict.
\end{remark}

To apply Theorem \ref{thm-crlx}, we shall employ the following recursion found by Kauers and Paule \cite{Kauers-Paule} with a computer algebraic system, and independently obtained by Moll \cite{Moll2007} via the WZ-method \cite{WZ}.

\begin{theorem}$($\cite[Eq. (8)]{Kauers-Paule}$)$\label{Thm:rec-dim-im}
For $m\ge 2$ and $1\le i \le m-1$, there holds
\begin{align}
  d_i(m)
=\frac{8m^2-8m-4i^2+3}{2m(m-i)}d_i(m-1)
      -\frac{(4m-5)(4m-3)(m-1+i)}{4m(m-1)(m-i)}d_i(m-2).\label{eq:rec-dim-m}
\end{align}
\end{theorem}

The following lower bound for the ratio $d_i(m+1)/d_i(m)$ derived by Zhao \cite{Zhao2023} is crucial to our proof of Theorem \ref{thm:st-r-lvx}.

\begin{theorem}$($\cite[Theorem 2.1]{Zhao2023}$)$\label{thm:lb-ratdlm1m}
For any $m\ge 2$ and $1\le i \le m-1$,we have
\begin{align}\label{ineq:lb-dlm1m}
\frac{d_i(m+1)}{d_i(m)}>L(m,i),
\end{align}
where
\begin{align}\label{eq:def-Lml}
L(m,i)=\frac{4m^2+7m-2i^2+3}{2(m+1)(m-i+1)}
  +\frac{i\sqrt{4i^4+8i^2m+5i^2+m}}{2(m+1)(m-i+1)\sqrt{m+i^2}}.
\end{align}
\end{theorem}

We are now ready to present a proof of Theorem \ref{thm:st-r-lvx}.
\begin{proof}[Proof of Theorem \ref{thm:st-r-lvx}]
Fix $i\ge 1$ throughout this proof. To apply Theorem \ref{thm-crlx}, let $S_n=d_i(n)$ and set $N=i$. By Theorem \ref{Thm:rec-dim-im}, we have
\begin{align*}
 d_i(n)=\mathfrak{a}(n)d_i(n-1)+\mathfrak{b}(n)d_i(n-2),
\end{align*}
for $n\ge i+2$,
where
\begin{align}\label{defi:an&bn}
\mathfrak{a}(n)=\frac{8n^2-8n-4i^2+3}{2n(n-i)},\qquad
\mathfrak{b}(n)=-\frac{(4n-5)(4n-3)(n-1+i)}{4n(n-1)(n-i)}.
\end{align}
Clearly, $\mathfrak{a}(n)>0$ and $\mathfrak{b}(n)<0$ for $n\ge i+2$.

For any $n\ge i+2$, we first prove the conditions $(i)$ and $(ii)$ in  Theorem \ref{thm-crlx}.
For this purpose, let
\begin{align}\label{eq:g(n)6}
 g(n)=L(n-1,i)
=\frac{4n^2-2i^2-n}{2n(n-i)}
 +\frac{i\sqrt{4i^4+8i^2n-3i^2+n-1}}{2n(n-i)\sqrt{i^2+n-1}}.
\end{align}
By Theorem \ref{thm:lb-ratdlm1m}, we have
\begin{align*}
g(n)< \frac{d_i(n)}{d_i(n-1)},
\end{align*}
for $n\ge i+2$. On the other hand, we see that
\begin{align}\label{ieq:g(n)>a(n)/2}
 g(n)-\frac{\mathfrak{a}(n)}{2}
=\frac{(6n-3)\Delta_1+2i\Delta_2}{4n(n-i)\Delta_1}>0,
\end{align}
for $n\ge i+2$, where
\begin{align}\label{defi:Delta12}
\Delta_1=\sqrt{i^2+n-1},\quad \Delta_2=\sqrt{4i^4+8i^2n-3i^2+n-1}.
\end{align}
So the inequalities in condition $(i)$ of Theorem \ref{thm-crlx} are confirmed.

We proceed to check the condition $(ii)$. Direct computation gives that
\begin{align}\label{eq:cond(ii)}
 &\ 4g^3(n)-3\mathfrak{a}(n)g^2(n)-\mathfrak{a}(n+1)\mathfrak{b}(n)\\
=&\ \frac{\Delta_1 (F_1(n-i-1)+\hat{F}_1)+\Delta_2 (F_2(n-i-1)i+\hat{F}_2)}
 {8(i^2+n-1)^{\frac{3}{2}}n^3(n-i)^3(n+1-i)(n^2-1)},\nonumber
\end{align}
where $\Delta_1$, $\Delta_2$ are given by \eqref{defi:Delta12}, and
\begin{align*}
 F_1 &
=240n^7+(240i^2+256i-364)n^6+(160i^4+256i^3-164i^2-32i+32)n^5\\
 &\quad +(280i^4+352i^3+676i^2-392i+166)n^4+32(2n^2-i^2)i^6n^2\\
 &\quad +(64i^6+448i^5+1004i^4+104i^3-64i^2+256i-83)n^3\\
 &\quad  +(128i^7+456i^6+1168i^5+798i^4+984i^3-203i^2-18i+9)n^2\\
 &\quad +(64i^8+576i^7+1568i^6+2440i^5+1596i^4+442i^3+103i^2)n\\
 &\quad  +64i^9+672i^8+2208i^7+3996i^6+3748i^5+2251i^4+560i^3+79i^2,\\
\hat{F}_1 & =64i^{10}+704i^9+2912i^8+6152i^7+7796i^6+6074i^5+2736i^4+648i^3+70i^2,\\
 F_2 &
=144n^6+36n^5+18(6i^2+16i-13)n^4+(288i^3+202i^2+72i-18)n^3\\
 &\quad +(144n^2-40i^2)i^2n^3+(16i^6+232i^4+504i^3+400i^2-180i+90)n^2\\
 &\quad +(32i^6+208i^5+800i^4+764i^3+230i^2+36i-18)n\\
 &\quad +4i^2(8i^5+56i^4+242i^3+395i^2+281i+44),\\
\hat{F}_2 & =4i^3(8i^6+68i^5+294i^4+643i^3+670i^2+315i+54).
\end{align*}

Clearly, $\Delta_1$, $\Delta_2$, $\hat{F}_1$, $\hat{F}_2$ and the denominator of the right-hand side of \eqref{eq:cond(ii)} are positive for $n\ge i+2$. It is easy to see that $F_1>0$, $F_2>0$ for $n\ge i+2$. It follows that \eqref{eq:cond(ii)} is positive for $n\ge i+2$, which leads to the condition $(ii)$.

It remains to prove that the conditions $(iii)$ in Theorem \ref{thm-crlx} holds for $n\ge i+2$. For convenience, define a function
\begin{align}\label{defi:h(x)}
 h(x):
=x^4-\mathfrak{a}(n)x^3-\mathfrak{a}(n+1)\mathfrak{b}(n)x
  -\mathfrak{b}(n)\mathfrak{b}(n+1),
\end{align}
for $x\in\mathbb{R}$.
By computation we get
\begin{align}\label{eq:cond(iii)hgn}
 h(g(n))
=\frac{G_1+\Delta_1 \Delta_2 G_2}
 {16(n^2-1)(n+1-i)(i^2+n-1)^2 n^4 (n-i)^4},
\end{align}
where $\Delta_1$, $\Delta_2$ are given by \eqref{defi:Delta12}, and
{\small
\begin{align*}
 G_1 &
=(512i-128)n^{10}+(1024i^3-768i^2-1664i+400)n^9\\
&\quad\ +(512i^5-928i^4-2304i^3+2880i^2+1392i-328)n^8\\
&\quad\ +(-352i^6-864i^5+4424i^4-192i^3-3584i^2+776i-173)n^7\\
&\quad\ +(-64i^8-160i^7+2840i^6-1896i^5-6252i^4+4448i^3+922i^2-1765i+410)n^6\\
&\quad\ +(64i^9+688i^8-1528i^7-4340i^6+6276i^5+2279i^4-4694i^3+1440i^2+885i-224)n^5\\
&\quad\ +(64i^{10}-496i^9-1336i^8+4396i^7+1592i^6-6645i^5+1812i^4+2068i^3\\
&\quad\qquad -1075i^2-139i+46)n^4\\
&\quad\ +(-64i^{11}-80i^{10}+1320i^9+332i^8-4396i^7+1554i^6+3155i^5\\
&\quad\qquad  -1916i^4-305i^3+149i^2+3i-3)n^3\\
&\quad\ +(144i^{11}+48i^{10}-1412i^9+672i^8+2152i^7-1615i^6-579i^5+599i^4-54i^3+45i^2)n^2\\
&\quad\ +(-192i^{11}+88i^{10}+740i^9-551i^8-537i^7+471i^6-20i^5+i^4+9i^3-9i^2)n\\
&\quad\ +104i^{11}-104i^{10}-189i^9+189i^8+66i^7-66i^6+19i^5-19i^4\\
G_2 &
=240in^8+(240i^3+16i^2-604i)n^7+(32i^5+16i^4-676i^3+76i^2+396i)n^6\\
&\quad +(-32i^6-280i^5+276i^4+864i^3-392i^2+134i)n^5\\
&\quad +(-32i^7+184i^6+564i^5-932i^4-310i^3+482i^2-249i)n^4\\
&\quad +(32i^8+40i^7-492i^6-112i^5+914i^4-379i^3-191i^2+92i)n^3\\
&\quad +(-72i^8-24i^7+484i^6-276i^5-325i^4+296i^3+9i^2-9i)n^2\\
&\quad +(96i^8-44i^7-208i^6+159i^5+29i^4-32i^3)n-52i^8+52i^7+49i^6-49i^5+3i^4-3i^3.
\end{align*}
}%

Clearly, the denominator of \eqref{eq:cond(iii)hgn} is positive for $n\ge i+2$. It suffices to show that
\begin{align}\label{ieq:c1}
 G_1+\Delta_1 \Delta_2 G_2>0, \qquad \mbox{ for } n\ge i+2.
\end{align}
Observe that
\begin{align}\label{ieq:delta12}
 \Delta_1\Delta_2
=&\ \sqrt{i^2+n-1}\sqrt{(4i^2+1)(i^2+n-1)+4i^2n}\nonumber\\
=&\ 2i(i^2+n-1)\sqrt{1+\frac{1}{4i^2}+\frac{n}{i^2+n-1}}\nonumber\\
>&\ 2i(i^2+n-1),
\end{align}
for $n\ge i+2$.
Also note that
\begin{align*}
G_2=i(n-i-1)H_1+H_2,
\end{align*}
where
\begin{align*}
 H_1 &
=240n^7+(240i^2+256i-364)n^6+(256i^3-180i^2-32i+32)n^5\\
&\quad\ +(8i^4+352i^3+652i^2-392i+166)n^4+32(n^2-i^2)i^4n^3\\
&\quad\ +(192i^5+924i^4+72i^3-50i^2+256i-83)n^3\\
&\quad\ +(200i^6+624i^5+884i^4+936i^3-173i^2-18i+9)n^2\\
&\quad\ +(128i^7+800i^6+1992i^5+1544i^4+438i^3+105i^2)n\\
&\quad\ +128i^8+1024i^7+2748i^6+3328i^5+2141i^4+572i^3+73i^2,\\
 H_2 &
=2i^3(64i^7+576i^6+1860i^5+3064i^4+2759i^3+1332i^2+324i+35).
\end{align*}
Clearly, $H_1>0$, $H_2>0$, and hence $G_2>0$ for $n\ge i+2$. It follows from \eqref{ieq:delta12} that
\begin{align}\label{ieq:B1B2>}
 G_1+\Delta_1 \Delta_2 G_2
>G_1+2i(i^2+n-1) G_2.
\end{align}

It remains to show that
\begin{align}\label{ieq:rem}
 G_1+2i(i^2+n-1) G_2>0,\quad n\ge i+2.
\end{align}
To this end, we rewrite the left hand side as
\begin{align*}
 G_1+2i(i^2+n-1) G_2
=(n-i-2)K_1+K_2,
\end{align*}
where
{\small
\begin{align*}
 K_1 &
=(512i-128)n^9+(1024i^3+224i^2-768i+144)n^8+(512i^5+1056i^4+872i^2-40)n^7\\
&\quad +(704i^6+1280i^5+3496i^4+800i^3+160i^2+736i-253)n^6\\
&\quad +(512i^7+3552i^6+4832i^5+5412i^4+5272i^3+1454i^2-546i-96)n^5\\
&\quad +(576i^8+4032i^7+11012i^6+18152i^5+16295i^4+9052i^3+3036i^2-303i-416)n^4\\
&\quad +(512i^9+5120i^8+20256i^7+39796i^6+50610i^5+42818i^4+21862i^3\\
&\quad\qquad   +5376i^2-1161i-786)n^3\\
&\quad +(512i^{10}+6272i^9+30476i^8+79692i^7+130670i^6+144333i^5\\
&\quad \qquad +107116i^4+49195i^3+9538i^2-3105i-1575)n^2\\
&\quad +(512i^{11}+7296i^{10}+42912i^9+140724i^8+290172i^7+405520i^6+395929i^5\\
&\quad\qquad +263352i^4+107856i^3+16034i^2-7785i-3150)n\\
&\quad +512i^{12}+8320i^{11}+57504i^{10}+226576i^9+571579i^8+985899i^7\\
&\quad\qquad +1206960i^6+1055138i^5+634619i^4+231755i^3+24274i^2-18720i-6300,\\
 K_2 &
=512i^{13}+9344i^{12}+74144i^{11}+341584i^{10}+1024744i^9+2129044i^8+3178732i^7\\
&\quad +3469084i^6+2744908i^5+1500980i^4+487784i^3+29828i^2-43740i-12600.
\end{align*}
}%

It is straightforward to check that $K_1>0$, $K_2>0$ for $n\ge i+2$, and hence
\begin{align*}
G_1+2i(i^2+n-1) G_2>0,\quad n\ge i+2.
\end{align*}
So \eqref{ieq:rem} is proved.
Combining \eqref{eq:cond(iii)hgn}, \eqref{ieq:B1B2>} and \eqref{ieq:rem}, we have $h(g(n))>0$ for $n\ge i+2$, that is, the condition $(iii)$ in Theorem \ref{thm-crlx} holds for $n\ge i+2$. It follows from Theorem \ref{thm-crlx} and Remark \ref{rmk} that \eqref{ieq:ra-lcv-1} holds strictly for $S_n=d_i(n)$. Therefore, we obtain the strict ratio-log-convexity of the transposed Boros-Moll sequence $\{d_i(m)\}_{m\ge i}$ for each $i\ge 1$.
\end{proof}

\section{Proof of Theorem \ref{Thm:Briggs-ieq-tr}}\label{Sec:4}

In this section we proceed to prove the Briggs inequality for the transposed Boros-Moll sequence $\{d_i(m)\}_{m\ge i}$ for $i\ge 1$.
Comparing with the proof of Theorem \ref{Thm:Briggs-ieq}, it is natural to consider the equivalent form of \eqref{ieq:Briggs-trBM}, that is, for $i\ge 1$ and $m\ge i+1$,
\begin{align}\label{ieq:Briggs-trBM-r}
 (v_i(m+1)-1)v_i(m)-1+\frac{1}{v_i(m)}>0,
\end{align}
where $v_i(m)=(d_{i}(m-1)d_{i}(m+1))/d_i^2(m)$.
However, numerical experiment shows that the left-hand side of \eqref{ieq:Briggs-trBM-r} trends to zero very fast, and the known bounds for $v_i(m)$ are not sufficiently sharp for the proof of \eqref{ieq:Briggs-trBM-r}. In order to prove Theorem \ref{Thm:Briggs-ieq-tr}, we find the following sufficient conditions for the Briggs inequality.

\begin{theorem}\label{thm:2-r-Briggs}
Let $\{a_n\}_{n\ge 0}$ be a sequence with positive numbers. Let $N_0$ be a positive integer. If the following conditions are satisfied:
\begin{itemize}
\item[(i)] $\{a_n\}_{n\ge N_0}$ is strictly log-concave,
\item[(ii)] $\{a_n\}_{n\ge N_0}$ is strictly ratio-log-convex,
\end{itemize}
then the Briggs inequality  holds for $\{a_n\}_{n\ge N_0}$. That is, for $n\ge N_0+1$,
\begin{align}\label{ieq:Briggs}
a_n^2(a_n^2 - a_{n-1}a_{n+1})> a_{n-1}^2(a_{n+1}^2 - a_n a_{n+2}).
\end{align}
Moreover, if $\{a_n\}_{n\ge N_0}$ is strictly $2$-log-convex and condition $(ii)$ holds, then the Briggs inequality \eqref{ieq:Briggs} holds for $n\ge N_0+2$.
\end{theorem}

\begin{proof}
Let $\{a_n\}_{n\ge 0}$ be a sequence with $a_n>0$ for all $n$. Let $N_0>0$ be an integer.
We first prove the first result that conditions $(i)$ and $(ii)$ imply \eqref{ieq:Briggs} for $n\ge N_0+1$.
By condition $(i)$, the sequence $\{a_n\}_{n\ge N_0}$ is strictly log-concave, that is, for $n\ge N_0+1$,
\begin{align}\label{ieq:st-lg-cv}
 a_n^2-a_{n-1}a_{n+1}>0.
\end{align}
Then we have for $n\ge N_0+1$,
\begin{align}\label{ieq:st-lg-cv1}
 \frac{a_n^2}{a_{n-1}^2}>\frac{a_{n+1}^2}{a_n^2}.
\end{align}

On the other hand, by condition (ii), the sequence $\{a_n\}_{n\ge N_0}$ is strictly ratio-log-convex, implying that
\begin{align}\label{ieq:r-lg-cv}
 \frac{a_{n-2}a_{n}}{a_{n-1}^2}
<\frac{a_{n-1}a_{n+1}}{a_n^2},
\quad n\ge N_0+2.
\end{align}
Hence for $n\ge N_0+1$,
\begin{align}\label{ieq:r-lg-cv2}
 \frac{a_{n-1}a_{n+1}}{a_{n}^2}
<\frac{a_{n}a_{n+2}}{a_{n+1}^2}.
\end{align}

It follows from \eqref{ieq:r-lg-cv2} and \eqref{ieq:st-lg-cv} that
\begin{align}\label{ieq:r1}
 1-\frac{a_{n-1}a_{n+1}}{a_{n}^2}
>1-\frac{a_{n}a_{n+2}}{a_{n+1}^2}
>0,
\end{align}
for $n\ge N_0+1$. Combining \eqref{ieq:st-lg-cv1} and \eqref{ieq:r1}, we obtain
\begin{align}\label{ieq:r2}
 \frac{a_n^2}{a_{n-1}^2}\left(1-\frac{a_{n-1}a_{n+1}}{a_{n}^2}\right)
>\frac{a_{n+1}^2}{a_n^2}\left(1-\frac{a_{n}a_{n+2}}{a_{n+1}^2}\right).
\end{align}
Multiply $a_n^2 a_{n-1}^2$ on both sides of \eqref{ieq:r2}. Then we have \eqref{ieq:Briggs} for $n\ge N_0+1$, and hence the first result is proved.

We proceed to show the second statement that the strict $2$-log-convexity of $\{a_n\}_{n\ge N_0}$ together with condition $(ii)$ imply the Briggs inequality \eqref{ieq:Briggs} for $n\ge N_0+2$.
Assume that $\{a_n\}_{n\ge N_0}$ is strictly $2$-log-convex, that is, for $n\ge N_0+2$,
\begin{align}\label{ieq:2-lg-cvx}
\frac{a_n^2-a_{n-1}a_{n+1}}{a_{n-1}^2-a_{n-2}a_{n}}
<\frac{a_{n+1}^2-a_{n}a_{n+2}}{a_n^2-a_{n-1}a_{n+1}}.
\end{align}
To make the proof more concise, let
\begin{align}\label{defi:cn}
c_n=\frac{a_{n-1}a_{n+1}}{a_n^2},\quad n\ge 1.
\end{align}
Thus, \eqref{ieq:2-lg-cvx} can be rewritten as
\begin{align}\label{ieq:2-lg-cv-re}
\frac{1-c_n}{1-c_{n-1}}
<c_n^2\frac{1-c_{n+1}}{1-c_n},
\end{align}
for $n\ge N_0+2$.
Clearly, $\{a_n\}_{n\ge N_0}$ is strictly log-convex. Then we have $c_n>1$ for $n\ge N_0+1$.
So, the inequality \eqref{ieq:2-lg-cv-re} is equivalent to
\begin{align}\label{ieq:2-lg-cv-re2}
(c_{n+1}-1)c_n^2>\frac{(c_n-1)^2}{c_{n-1}-1},\quad n\ge N_0+2.
\end{align}

On the other hand, by condition (ii), we have \eqref{ieq:r-lg-cv}, which can be restated as
\begin{align}\label{ieq:ratio-lg-cv-re}
 c_{n-1}<c_n, \quad n\ge N_0+2.
\end{align}

Combining \eqref{ieq:2-lg-cv-re2} and \eqref{ieq:ratio-lg-cv-re}, one can easily obtain that for $n\ge N_0+2$,
\begin{align}\label{ieq:Briggs-cn>}
(c_{n+1}-1)c_n^2>\frac{(c_n-1)^2}{c_n-1}=c_n-1.
\end{align}
Substituting \eqref{defi:cn} into \eqref{ieq:Briggs-cn>} leads to the desired inequality \eqref{ieq:Briggs} for $n\ge N_0+2$.
\end{proof}

To prove Theorem \ref{Thm:Briggs-ieq-tr}, we also need the following result due to Jiang and Wang \cite{Jiang-Wang}, where the log-concavity of the sequences $\{d_i(m)\}_{m\ge i}$ for $i\ge 1$ was established.

\begin{theorem}$($\cite[Theorem 3.1]{Jiang-Wang}$)$\label{thm:J-W}
For any $i\ge 1$, the sequence $\{d_i(m)\}_{m \ge i}$ is log-concave.
\end{theorem}
Notice that by the proof of Theorem \ref{thm:J-W}, it is clear that the sequence $\{d_i(m)\}_{m \ge i}$ is strictly log-concave for any $i\ge 1$.
We would like to point out that Zhao \cite{Zhao2023-2} obtained an inequality in studying the extended reverse ultra log-concavity of the transposed Boros-Moll sequences $\{d_i(m)\}_{m \ge i}$, which implies the strict log-concavity of $\{d_i(m)\}_{m \ge i}$ for $i\ge 1$.
\begin{theorem}$($\cite[Theorem 3.2]{Zhao2023-2}$)$\label{thm:lbdlm2}
For each $i\ge 0$ and $m\ge i+1$, we have
\begin{align}\label{ineq:lbdlm2}
\frac{d_i^2(m)}{d_i(m-1)d_i(m+1)}
> \frac{(m-i+1)m^3}{(m-i)(m+1)(m^2+1)}.
\end{align}
\end{theorem}
Denote by $R(i,m)$ the right-hand side of \eqref{ineq:lbdlm2}. Clearly, $R(1,m)=m^4/(m^4-1)>1$ for $m\ge 2$, and it is easily checked that $R(i,m)>(m^2+1)/m^2>1$ for $i\ge 2$ and $m\ge i+1$.

We are now in a position to give a proof of Theorem \ref{Thm:Briggs-ieq-tr}.

\begin{proof}[Proof of Theorem \ref{Thm:Briggs-ieq-tr}]
Fix $i\ge 1$. To apply Theorem \ref{thm:2-r-Briggs}, set $a_n=d_i(n)$ and let $N_0=i$. By Theorem \ref{thm:J-W} or Theorem \ref{thm:lbdlm2}, the sequences $\{d_i(n)\}_{n \ge i}$ are strictly log-concave for $i\ge 1$. By Theorem \ref{thm:st-r-lvx}, the transposed Boros-Moll sequences $\{d_i(n)\}_{n\ge i}$ are strictly ratio-log-convex for any $i\ge 1$. Thus, for any given $i\ge 1$, it follows from Theorem \ref{thm:2-r-Briggs} that \eqref{ieq:Briggs} holds for $n\ge i+1$, or equivalently, \eqref{ieq:Briggs-trBM} holds for $m\ge i+1$.

It remains to prove the reverse Briggs inequality for $i=0$. That is, for $m\ge 1$,
\begin{align}\label{ieq:Briggs-trBM-0}
  d_{0}^2(m)(d_{0}^2(m)-d_{0}(m-1)d_{0}(m+1))
<d_{0}^2(m-1)(d_{0}^2(m+1)-d_{0}(m)d_{0}(m+2)).
\end{align}
Observe that \eqref{ieq:Briggs-trBM-0} holds if and only if
\begin{align}\label{ieq:B-trBM-r}
 \delta_0(m):=
 r_{0}^2(m)\left(1-\frac{r_0(m+1)}{r_0(m)}\right)
 -r_0^2(m+1)\left(1-\frac{r_0(m+2)}{r_0(m+1)}\right)
 <0,
\end{align}
for $m\ge 1$, where $r_0(m)=d_0(m)/d_0(m-1)$.
Recall that for $i\ge 0$ and $m\ge i$, Kauers and Paule \cite[Eq. (6)]{Kauers-Paule} showed a recurrence relation as follows.
\begin{align}
  d_i(m+1)
=\frac{m+i}{m+1}d_{i-1}(m)+\frac{4m+2i+3}{2(m+1)}d_i(m).\label{eq:reclm1}
\end{align}
By setting $i=0$ in \eqref{eq:reclm1}, we have
\begin{align*}
r_0(m)=\frac{4m-1}{2m},\quad m\ge 1.
\end{align*}
It is easily checked that
\begin{align*}
 \delta_0(m)=-\frac{8m^2+5m-2}{4m^2(m+1)^2(m+2)}<0,
\end{align*}
for $m\ge 1$, and hence \eqref{ieq:B-trBM-r} is true. This completes the proof.
\end{proof}

\section{Strict log-convexity of $\{\sqrt[n]{d_i(i+n)}\}_{n\ge 1}$}\label{Sec:5}

In this section, we show the strict log-convexity of the sequence $\{\sqrt[n]{d_i(i+n)}\}_{n\ge 1}$ for $i\ge 1$, by applying Theorem \ref{thm:st-r-lvx} and a sufficient condition given by Chen, Guo and Wang \cite{CGW2014}. Note that Zhao \cite{Zhao2023-2} proved an asymptotic result of this property for $i\ge 136$.
\begin{theorem}$($\cite[Theorems 7.1 \& 7.2]{Zhao2023-2}$)$\label{thm:nsqr-dlm-Lv}
The sequence $\{\sqrt[n]{d_i(i+n)}\}_{n\ge 1}$ is strictly log-concave for $i=0$, and is strictly log-convex for each $1\le i \le 135$.
For $i\ge 136$, the sequences $\{\sqrt[n]{d_i(i+n)}\}_{n\ge i^2}$ are strictly log-convex.
\end{theorem}

The main result of this section is as follows.
\begin{theorem}\label{thm:nsqr-dim-Lv}
The sequence $\{\sqrt[n]{d_i(i+n)}\}_{n\ge 1}$ is strictly log-convex for each $i\ge 1$. That is, for any $i\ge 1$ and $n\ge 1$, we have
\begin{align}
\frac{\sqrt[n+1]{d_i(i+n+1)}}{\sqrt[n]{d_i(i+n)}}
<\frac{\sqrt[n+2]{d_i(i+n+2)}}{\sqrt[n+1]{d_i(i+n+1)}}.
\end{align}
\end{theorem}

Chen, Guo and Wang \cite{CGW2014} established the following criterion which indicates that ratio log-convexity of a sequence $\{S_n\}$ together with certain initial condition imply log-convexity of the sequence $\{\sqrt[n]{S_n}\}$.

\begin{theorem}$($\cite[Theorem 3.6]{CGW2014}$)$\label{thm:cgw3.6}
Let $\{S_n\}_{n\ge 0}$ be a positive sequence. If the sequence $\{S_n\}_{n\ge N}$ is ratio log-convex and
\begin{align}\label{eq:cgw3.6}
  \frac{\sqrt[N+1]{S_{N+1}}}{\sqrt[N]{S_{N}}}
 <\frac{\sqrt[N+2]{S_{N+2}}}{\sqrt[N+1]{S_{N+1}}}
\end{align}
for some positive integer $N$, then the sequence $\{\sqrt[n]{S_n}\}_{n\ge N}$ is strictly log-convex.
\end{theorem}

We are now ready to present a proof of Theorem \ref{thm:nsqr-dim-Lv}.
\begin{proof}[Proof of Theorem \ref{thm:nsqr-dim-Lv}]

For $i\ge 1$, setting $S_n=d_i(i+n)$ and $N=1$ in Theorem \ref{thm:cgw3.6}. By Theorem \ref{thm:st-r-lvx}, the sequence $\{d_i(i+n)\}_{n\ge 0}$ is strictly ratio log-convex for each $i \ge 1$. It suffices to verify the condition \eqref{eq:cgw3.6}, or equivalently,
\begin{align}\label{defi:mf(r(i))}
 \frac{d_i(i+2)}{d_i(i+1)\cdot \sqrt[3]{d_i(i+3)}}<1,\quad i\ge 1.
\end{align}
For convenience, denote by $\mathfrak{r}(i)$ the left-hand side of \eqref{defi:mf(r(i))}.

Notice that it was proved in \cite[Lemma 7.5]{Zhao2023-2} that the sequence $\{\mathfrak{r}(i)\}_{i\ge 0}$ is strictly decreasing. That is, for $i\ge 0$,
\begin{align}\label{ieq:dlmsqr>}
\frac{d_i(i+2)}{d_i(i+1)\sqrt[3]{d_i(i+3)}}
>\frac{d_{i+1}(i+3)}{d_{i+1}(i+2)\sqrt[3]{d_{i+1}(i+4)}}.
\end{align}
To be self-contained, we make a brief overview of the proof for \eqref{ieq:dlmsqr>}.
Since $d_i(m)>0$ for $m\ge i\ge 0$, the inequality \eqref{ieq:dlmsqr>} holds if and only if
\begin{align}\label{ieq:dlmsqr>re}
\frac{d_i(i+2)d_{i+1}(i+2)}{d_i(i+1)d_{i+1}(i+3)}
>\frac{\sqrt[3]{d_i(i+3)}}{\sqrt[3]{d_{i+1}(i+4)}}.
\end{align}
Applying \eqref{eq:B-M-seq}, it is easy to obtain the expressions of the terms in \eqref{ieq:dlmsqr>re}. For example,
\begin{align*}
 d_i(i+1)
=\frac{(2i+3)(2i+1)(2i)!}{2^{i+1}(i+1)(i !)^2}.
\end{align*}
Denote by $V_1$ and $V_2$, respectively, the LHS and RHS of \eqref{ieq:dlmsqr>re}. We have checked that
\begin{align*}
 V_1^3-V_2^3
=\frac{(i+1)\sum_{k=0}^{18}s_k i^k}{(i+2)^6 (2i+3)^3 (4i^2+26i+43)^3 (4i^2+30i+59)(2i+9)(2i+5)},
\end{align*}
where $s_k>0$ are integers for $k=0,1,\ldots,18$.
Clearly, $V_1^3-V_2^3>0$ for $i\ge 0$, which leads to \eqref{ieq:dlmsqr>re}, as well as \eqref{ieq:dlmsqr>}.

Using \eqref{eq:B-M-seq} again, it is easily verified that
\begin{align*}
 \mathfrak{r}(1)
=\frac{d_1(3)}{d_1(2)\cdot \sqrt[3]{d_1(4)}}
=\frac{43\cdot 885^\frac{2}{3}\cdot 32^\frac{1}{3}}{13275}
=0.9+\epsilon_1,\quad 0<\epsilon_1<10^{-1}.
\end{align*}
Thus, $\mathfrak{r}(1)<1$. By \eqref{ieq:dlmsqr>}, we see that $\mathfrak{r}(i)<1$ for all $i\ge 1$, and hence \eqref{defi:mf(r(i))} is proved. That is, for any $i\ge 1$, the condition in \eqref{eq:cgw3.6} is satisfied for $N=1$. It follows from Theorem \ref{thm:cgw3.6} that the sequence  $\{\sqrt[n]{d_i(i+n)}\}_{n\ge 1}$ is strictly log-convex for each $i \ge 1$. The proof is complete.
\end{proof}

\section*{Acknowledgments}
This work was partially supported by the Fundamental Research Funds for the Central Universities.

\end{document}